\documentclass{amsart}

\usepackage{mathrsfs}
\usepackage[dvips]{graphicx}
\usepackage{amsmath}
\usepackage[latin1]{inputenc}
\usepackage{amsfonts}
\usepackage{amssymb}
\usepackage{graphicx,epsfig}
\usepackage{amsthm}

\usepackage{amscd}
\usepackage{dsfont}
\usepackage[all,dvips]{xy}
\usepackage{fancyhdr}
\usepackage[T1]{fontenc}

\input cyracc.def

\newtheorem{theorem}{Theorem}[section]

\newtheorem{proposition}[theorem]{Proposition}

\theoremstyle{definition}
\newtheorem{definition}[theorem]{Definition}
\newtheorem{example}[theorem]{Example}

\theoremstyle{remark}
\newtheorem{remark}[theorem]{Remark}

\allowdisplaybreaks

\numberwithin{equation}{section}

\begin{document}

\title{From quantum quasi-shuffle algebras to braided Rota-Baxter algebras}

\author{Run-Qiang Jian}
\address{School of Computer Science, Dongguan University
of Technology, 1, Daxue Road, Songshan Lake, 523808, Dongguan, P.
R. China}
\email{jian.math@gmail.com}
\thanks{}


\date{}

\dedicatory{To Emma Yin}


\begin{abstract}
In this letter, we use quantum quasi-shuffle algebras to construct
Rota-Baxter algebras, as well as tridendriform algebras. We also
propose the notion of braided Rota-Baxter algebras, which is the
relevant object of Rota-Baxter algebras in a braided tensor
category. Examples of such new algebras are provided by using
quantum multi-brace algebras in a category of Yetter-Drinfeld
modules.
\end{abstract}

\maketitle
\section{Introduction}
Rota-Baxter operators originate from a work on probability theory
 more than four decades ago. In the paper \cite{Bax}, G. Baxter deduced many
well-known identities in the theory of fluctuations for random
variables from a simple relation of operators on a commutative
algebra. Later, based on this work and others, G.-C. Rota
introduced the notion of Baxter algebra, which is called
Rota-Baxter algebra nowadays, in his fundamental papers
\cite{Rot}. Since then, this new algebraic object is investigated
by many mathematicians with various motivations. Besides their own
interest in mathematics, Rota-Baxter algebras have many
significant applications in mathematical physics. For instance,
they play essential role in the Hopf algebraic approach to the
Connes-Kreimer theory of renormalization in perturbative quantum
field theory. There the theory of non-commutative Rota-Baxter
algebras with idempotent Rota-Baxter operator is used to provide
an algebraic setting for the formulation of renormalization (cf.
\cite{CK1}, \cite{CK2}, \cite{EGK1}, \cite{EGK2}). Rota-Baxter
algebras are also related to the works of pre-Poisson algebras
\cite{A} and Loday-type algebras \cite{Fa}.

There is a seminal but implicit connection between Rota-Baxter
algebras and another important kind of algebras. This is the
quasi-shuffle algebra. Quasi-shuffle algebras first arose in the
study of the cofree irreducible Hopf algebra built on an arbitrary
associative algebra by K. Newman and D. E. Radford (\cite{NR}). In
2000, they were rediscovered independently by M. E. Hoffman in an
inductive form (\cite{Hof2}). In the same year, L. Guo and W.
Keigher (\cite{Guo1} and \cite{GK}) defined a new algebra
structure called mixable shuffle algebra and used it to construct
free commutative Rota-Baxter algebras. Nevertheless, the mixable
shuffle algebra is equivalent to the quasi-shuffle algebra in some
sense. Recently, quasi-shuffle algebras have many important
applications in other areas of mathematics, such as multiple zeta
values (\cite{Hof1} and \cite{IKZ}), Rota-Baxter algebras
(\cite{EG}), and commutative tridendriform algebras (\cite{Lod}).
They also appear in mathematical physics. For instance, D. Kreimer
used quasi-shuffle algebras to study shuffle identities between
Feynman graphs (\cite{K}).

On the other hand, after the birth of quantum groups,
mathematicians and mathematical physicists begin to be interested
in braided tensor categories and the quantization of classical
algebraic structures. These braided or quantized objects not only
provide new subjects but also bring deeper interpretation of the
classical ones. They also provide new examples and tools in the
research of noncommutative geometry and conformal field theory
(cf. \cite{BK}). For the importance of quasi-shuffle algebras and
as a subsequent work of his quantum shuffles, M. Rosso constructed
the quantum quasi-shuffle algebra which is the quantization of
quasi-shuffle algebras. Some interesting properties and
applications of quantum quasi-shuffle algebras have been
discovered during the last few years (cf. \cite{JRZ} and
\cite{FR}). To our surprise, by applying a similar construction of
Guo and Keigher (\cite{GK}) to quantum quasi-shuffle algebras, we
can still provide Rota-Baxter algebras and tridendriform algebras.
This phenomenon does not happen usually since the complicated
action of the braiding makes quantum objects behave quite
differently from the usual ones. Based on this construction, we
can also construct idempotent Rota-Baxter operators. This
discovery extremely enlarges the families of Rota-Baxter algebras
and tridendriform algebras. Meanwhile, we consider the relevant
object of Rota-Baxter algebras in a braided tensor category. In
order to define such an object, we have to impose compatibility
between the braiding and the multiplication, as well as the
Rota-Baxter operator. Examples of these new algebras are
constructed from quantum multi-brace algebras in a category of
Yetter-Drinfeld modules.

This paper is organized as follows. In Section 2, we construct
Rota-Baxter algebras and tridendriform algebras by using quantum
quasi-shuffle algebras. In Section 3, we define the notion of
braided Rota-Baxter algebra and provide examples from quantum
multi-brace algebras in a category of Yetter-Drinfeld modules.

\noindent\textbf{Notation.} In this paper, we denote by
$\mathbb{K}$ a ground field of characteristic 0. All the objects
we discuss are defined over $\mathbb{K}$. For a vector space $V$,
we denote by $T(V)$ the tensor algebra of $V$, by $\otimes$ the
tensor product within $T(V)$, and by $\underline{\otimes}$ the one
between $T(V)$ and $T(V)$.

We denote by $\mathfrak{S}_{n}$ the symmetric group acting on the
set $\{1,2,\ldots,n\}$ and by $s_{i}$, $1\leq i\leq n-1$, the
standard generators of $\mathfrak{S}_{n}$ permuting $i$ and $i+1$.

A braiding $\sigma$ on a vector space $V$ is an invertible linear
map in $\mathrm{End}(V\otimes V)$ satisfying the quantum
Yang-Baxter equation on $V^{\otimes 3}$: $$(\sigma\otimes
\mathrm{id}_{V})(\mathrm{id}_{V}\otimes \sigma)(\sigma\otimes
\mathrm{id}_{V})=(\mathrm{id}_{V}\otimes \sigma)(\sigma\otimes
\mathrm{id}_{V})(\mathrm{id}_{V}\otimes \sigma).$$ A braided
vector space $(V,\sigma)$ is a vector space $V$ equipped with a
braiding $\sigma$. For any $n\in \mathbb{N}$ and $1\leq i\leq
n-1$, we denote by $\sigma_i$ the operator $\mathrm{id}_V^{\otimes
i-1}\otimes \sigma\otimes \mathrm{id}_V^{\otimes n-i-1}\in
\mathrm{End}(V^{\otimes n})$. For any $w\in \mathfrak{S}_{n}$, we
denote by $T^\sigma_w$ the corresponding lift of $w$ in the braid
group $B_n$, defined as follows: if $w=s_{i_1}\cdots s_{i_l}$ is
any reduced expression of $w$, then $T^\sigma_w=\sigma_{i_1}\cdots
\sigma_{i_l}$. This definition is well-defined (see, e.g., Theorem
4.12 in \cite{KT}).

We define $\beta:T(V)\underline{\otimes} T(V)\rightarrow
T(V)\underline{\otimes} T(V)$ by requiring that the restriction of
$\beta$ on $V^{\otimes i}\underline{\otimes} V^{\otimes j}$,
denoted by $\beta_{ij}$, is $T^\sigma_{\chi_{ij}}$ , where
\[\chi_{ij}=\left(\begin{array}{cccccccc}
1&2&\cdots&i&i+1&i+2&\cdots & i+j\\
j+1&j+2&\cdots&j+i&1& 2 &\cdots & j
\end{array}\right)\in \mathfrak{S}_{i+j},\] for any $i,j\geq 1$. For convenience, we
interpret $\beta_{0i}$ and $\beta_{i0}$ as the identity map of
$V^{\otimes i}$.

Let $(C,\Delta,\varepsilon)$ be a coalgebra. We denote
$\Delta^{(0)}=\mathrm{id}_C$, $\Delta^{(1)}=\Delta$, and
$\Delta^{(n+1)}=(\Delta^{(n)}\otimes \mathrm{id}_C)\circ\Delta$
recursively for $n\geq 1$. We adopt Sweedler's notation for
coalgebras and comodules: for any $c\in C$,
$$\Delta(c)=\sum c_{(1)}\otimes  c_{(2)},$$for a left $C$-comodule $(M,\rho)$ and any $ m\in M$,
$$\rho(m)=\sum m_{(-1)}\otimes  m_{(0)}.$$

\section{Rota-Baxter algebras coming from quantum quasi-shuffle algebras}

We start by recalling some basic notions. In this letter, by an
algebra we always mean an associative $\mathbb{K}$-algebra which
is not necessarily unital.

\begin{definition}Let $\lambda$ be an element in $\mathbb{K}$. A pair $(R,P)$ is
called a \emph{Rota-Baxter algebra of weight $\lambda$} if $R$ is
an algebra and $P$ is a linear endomorphism of $R$ satisfying that
for any $x,y\in R$,
$$P(x)P(y)=P(xP(y))+P(P(x)y)+\lambda P(xy).$$The map $P$ is called a \emph{Rota-Baxter operator}.\end{definition}

For a systematic and detailed introduction of these algebras, we
refer the readers to the book \cite{Guo2}.

By the works \cite{Guo1} and \cite{GK}, Rota-Baxter algebras are
closely related to mixable shuffle algebras. More precisely,
mixable shuffle algebras provide the free object in the category
of commutative Rota-Baxter algebras. Now we generalize this
construction to the quantized version of mixable shuffle algebras.
Our framework is the braided category. We need to work with the
relevant object of associative algebras in braided categories.

\begin{definition}Let $A=(A,m)$ be an algebra with product $m$, and $\sigma$ be a braiding on $A$. We call the triple $(A,m,\sigma)$ a \emph{braided algebra} if it satisfies the following conditions:
\[
\begin{split}
(\mathrm{id}_A\otimes m)\sigma_1\sigma_2&=\sigma( m\otimes
\mathrm{id}_A),\\[3pt] ( m\otimes
\mathrm{id}_A)\sigma_2\sigma_1&=\sigma(\mathrm{id}_A\otimes m).
\end{split}
\]
Moreover, if $A$ is unital and its unit $1_A$ satisfies that for
any $a\in A$,
\[
\begin{split}
\sigma(a\otimes 1_A)&=1_A\otimes a,\\[3pt] \sigma(1_A\otimes a)&=a\otimes
1_A,
\end{split}
\]then $A$ is called a \emph{unital braided algebra}.
\end{definition}

\begin{remark}1. Let $\tau$ be the usual flip map which switches two arguments $a\otimes b\mapsto b\otimes a$. Obviously, it is a braiding. For any algebra $(A,m)$, it is evident that the first two identities in the definition of braided algebras hold automatically for $m$ and $\tau$. Hence any associative algebra is a braided algebra with respect to $\tau$. On the other hand, for any braided vector space $(V,\sigma)$, there is a trivial
braided algebra structure on it whose multiplication is the
trivial one $m=0$.

2. The braided algebra structure is very crucial. Given a braided
vector space $(V,\sigma)$, it is not reasonable that there should
be a non-trivial braided algebra structure on $V$. For instance,
let $V$ be a vector space with basis $\{e_1,e_2\}$. Consider the
following two braidings on $V$:\[
\begin{split}
\sigma(e_{1}\otimes e_{1})&=e_{1}\otimes e_{1}, \\[3pt]
\sigma(e_{i}\otimes e_{2})&=qe_{2}\otimes e_{1},\\[3pt]
\sigma(e_{2}\otimes e_{1})&=qe_{1}\otimes
e_{2}+(1-q^{2})e_{2}\otimes e_{1},\\[3pt]
\sigma(e_{2}\otimes e_{2})&=e_{2}\otimes e_{2},
\end{split}
\] and $$\sigma'(e_{i}\otimes e_{j})=qe_{j}\otimes e_{i},\ \ \ \forall i,j,$$
where $q\in\mathbb{K}$ is nonzero and not equal to $\pm 1$.

We mention that $\sigma$ is of Hecke type. It comes originally
from the R-matrix of the quantum enveloping algebra $U_q
\mathfrak{sl}_2$. Both of these braidings play important roles in
the theory of quantum groups. But by an easy argument on the
structure coefficients of the multiplication, one can show that
the only product on $V$ which is compatible with $\sigma$ is just
the trivial one. The case is the same for $\sigma'$.

Since $\sigma$ is of Hecke type, we can ask the following
interesting question: Does there exist a non-trivial braided
algebra whose braiding is of Hecke type?

3. For any braided algebra $(A,m,\sigma)$, one can embed it into a
unital braided algebra
$(\widetilde{A},\widetilde{m},\widetilde{\sigma})$ in the
following way. First of all, we set
$\widetilde{A}=\mathbb{K}\oplus A$. Then we define the
multiplication $\widetilde{m}$ and the braiding
$\widetilde{\sigma}$ by: for any $\lambda,\mu\in \mathbb{K}$ and
$a,b\in A$
$$\widetilde{m}\big((\lambda+a)\otimes (\mu+b)\big)=\lambda\mu+\lambda\cdot b+\mu\cdot
a+m(a\otimes b), $$ and
$$\widetilde{\sigma}\big((\lambda+a)\otimes (\mu+b)\big)=\mu\otimes
\lambda+b\otimes \lambda+\mu\otimes a+\sigma(a\otimes b).
$$It is easy to verify that $(\widetilde{A},\widetilde{m},\widetilde{\sigma})$ is a braided algebra with unit $1\in \mathbb{K}$.
\end{remark}

Given a braided algebra $(A, m,\sigma)$, we define a map
$\Join_\sigma: T(A)\underline{\otimes} T(A)\rightarrow T(A)$ as
follows. For any $\lambda\in \mathbb{K}$ and $x\in T(A)$,
$$\lambda\Join_\sigma x=x \Join_\sigma \lambda=\lambda \cdot x.$$

For $i,j\geq 1$ and any $a_1,\ldots, a_i,b_1,\ldots, b_j\in A$,
$\Join_\sigma$ is defined recursively by
$$a_1\Join_\sigma b_1= a_1\otimes b_1+\sigma(a_1\otimes
b_1)+m(a_1\otimes b_1),$$
\begin{eqnarray*}\lefteqn{ a_1\Join_\sigma  (b_1\otimes
\cdots\otimes
b_j)}\\[3pt]
&=&a_1\otimes b_1\otimes \cdots\otimes  b_{j}+(\mathrm{id}_A\otimes\Join_{\sigma  (1,j-1)})(\beta_{1,1}\otimes\mathrm{id}_A^{\otimes  j-1})(a_1\otimes b_1\otimes \cdots\otimes  b_j)\\[3pt]
&&+m(a_1\otimes  b_1)\otimes b_2\otimes \cdots\otimes  b_j,
\end{eqnarray*}
\begin{eqnarray*}
\lefteqn{(a_1\otimes \cdots\otimes  a_i)\Join_\sigma  b_1}\\[3pt]
&=&a_1\otimes  \big((a_2\otimes \cdots\otimes  a_i)\Join_\sigma  b_1\big)+\beta_{i,1}(a_1\otimes \cdots\otimes  a_i\otimes  b_1)\\[3pt]
&&+(m\otimes\mathrm{id}_A^{\otimes  i-1})(\mathrm{id}_A\otimes
\beta_{i-1,1})(a_1\otimes \cdots\otimes a_i\otimes b_1),
\end{eqnarray*}
and\begin{eqnarray*}
\lefteqn{(a_1\otimes \cdots\otimes  a_i)\Join_\sigma  (b_1\otimes \cdots\otimes  b_j)}\\[3pt]
&=&a_1\otimes  \big((a_2\otimes \cdots\otimes  a_i)\Join_\sigma (b_1\otimes \cdots\otimes  b_{j})\big)\\[3pt]
&&+(\mathrm{id}_A\otimes  \Join_{\sigma  (i,j-1)})(\beta_{i,1}\otimes \mathrm{id}_A^{\otimes  j-1})(a_1\otimes \cdots\otimes  a_i\otimes  b_1\otimes \cdots\otimes  b_j)\\[3pt]
&&+(m\otimes \Join_{\sigma  (i-1,j-1)} )(\mathrm{id}_A\otimes
\beta_{i-1,1}\otimes \mathrm{id}_A^{\otimes  j-1})(a_1\otimes
\cdots\otimes a_i\otimes b_1\otimes \cdots\otimes  b_j),
\end{eqnarray*}
where $\Join_{\sigma (i,j)}$ denotes the restriction of
$\Join_\sigma $ on $A^{\otimes  i}\underline{\otimes} A^{\otimes
j}$.

Then $T_{\sigma}^{qsh}(A)=(T(A),\Join_\sigma)$ is an associative
algebra with unit $1\in \mathbb{K}$, and called the \emph{quantum
quasi-shuffle algebra} built on $(A, m,\sigma)$. It is the
quantization of the usual quasi-shuffle algebra. We observe that
$T_{\sigma}^{qsh}(A)$ is a filtered algebra with the filtration
$T_{\sigma}^{qsh}(A)^n=\bigoplus_{i=0}^n A^{\otimes i}$ (see
Proppsition 4.22 in \cite{JR}). For more information about quantum
quasi-shuffle algebras, one can see \cite{JRZ}.

Let $(A, m,1_A,\sigma)$ be a unital braided algebra. Then
manifestly $(A, \lambda\cdot m,\sigma)$ is a braided algebra for
any $\lambda\in \mathbb{K}$. We denote by $\Join_{\sigma,\lambda}$
the quantum quasi-shuffle product with respect to $(A,
\lambda\cdot m,\sigma)$. By Lemma 3 in \cite{Ba}, the space
$A\underline{\otimes} T_{\sigma}^{qsh}(A)$ is a unital associative
algebra with the product
$$\lozenge_{\sigma,\lambda}=(m\otimes
\Join_{\sigma,\lambda})(\mathrm{id}_A\otimes \beta \otimes
\mathrm{id}_{T(A)}).$$ We denote by
$\mathcal{R}_{\sigma,\lambda}(A)$ the pair $(A\underline{\otimes}
T_{\sigma}^{qsh}(A),\lozenge_{\sigma,\lambda})$. We can view
$\mathcal{R}_{\sigma,\lambda}(A)$ as $T^+(A)=\bigoplus_{i\geq
1}A^{\otimes i}$ at the level of vector spaces. We have two
products $\Join_{\sigma,\lambda}$ and $\lozenge_{\sigma,\lambda}$
on $T^+(A)$. This is an example of 2-braided algebras which
produce quantum multi-brace algebras (for the definitions, see
\cite{JR}). We define an endomorphism $P
:\mathcal{R}_{\sigma,\lambda}(A)\rightarrow\mathcal{R}_{\sigma,\lambda}(A)$
by
\[\begin{split}P (a_0\underline{\otimes}u)&=1_A\underline{\otimes}a_0\otimes u,\ \mathrm{if}\ u\in T^+(A),\\[3pt]
P (a_0\underline{\otimes}\nu)&=1_A\underline{\otimes}\nu\cdot
a_0,\ \mathrm{if}\ \nu \in \mathbb{K} .\end{split}\]

\begin{theorem}Under the assumptions above, the pair $(\mathcal{R}_{\sigma,\lambda}(A), P)$ is a Rota-Baxter algebra of weight $\lambda$.\end{theorem}
\begin{proof}Observe that $\beta(u\underline{\otimes}1_A)=1_A\underline{\otimes}u$ for any $u\in T(A)$. Therefore for any $a, b\in A$ and $x\in A^{\otimes i}$,$ y\in A^{\otimes j}$, we have
\begin{eqnarray*}\lefteqn{P\big((a\underline{\otimes}x)\lozenge_{\sigma,\lambda}P(b\underline{\otimes}y)\big)}\\[3pt]
&=&P\big((a\underline{\otimes}x)\lozenge_{\sigma,\lambda}(1_A\underline{\otimes}b\otimes
y)\big)=P\Big(a\underline{\otimes}\big(x\Join_{\sigma,\lambda}(b\otimes
y)\big)\Big)\\[3pt]
&=&1_A\underline{\otimes}a\otimes
\big(x\Join_{\sigma,\lambda}(b\otimes y)\big),\end{eqnarray*}
\begin{eqnarray*}\lefteqn{P\big(P(a\underline{\otimes}x)\lozenge_{\sigma,\lambda}(b\underline{\otimes}y)\big)}\\[3pt]
&=&P\big((1_A\underline{\otimes}a\otimes  x)\lozenge_{\sigma,\lambda}(b\underline{\otimes}y)\big)\\[3pt]
&=&1_A\underline{\otimes}\big((\mathrm{id}_A\otimes
\Join_{\sigma,\lambda,(i+1,j)})(\beta_{i+1,1}\otimes
\mathrm{id}_A^{\otimes j})(a\otimes  x\otimes b\otimes
y)\big),\end{eqnarray*}and
\begin{eqnarray*}\lefteqn{\lambda P\big((a\underline{\otimes}x)\lozenge_{\sigma,\lambda}(b\underline{\otimes}y)\big)}\\[3pt]
&=& 1_A\underline{\otimes}\big((\lambda\cdot m\otimes
\Join_{\sigma,\lambda})(\mathrm{id}_A\otimes \beta_{i1} \otimes
\mathrm{id}_A^{\otimes j})(a\otimes x\otimes b\otimes y)\big)
.\end{eqnarray*} By taking a summation, we get
\begin{eqnarray*}\lefteqn{P\big((a\underline{\otimes}x)\lozenge_{\sigma,\lambda}P(b\underline{\otimes}y)\big)+P\big(P(a\underline{\otimes}x)\lozenge_{\sigma,\lambda}(b\underline{\otimes}y)\big)+\lambda P\big((a\underline{\otimes}x)\lozenge_{\sigma,\lambda}(b\underline{\otimes}y)\big)}\\[3pt]
&=&1_A\underline{\otimes}\Big(a\otimes
\big(x\Join_{\sigma,\lambda}(b\otimes
y)\big)\\[3pt]
&&\ \ \ \ \ \ \ \ \ +(\mathrm{id}_A\otimes
\Join_{\sigma,\lambda,(i+1,j)})(\beta_{i+1,1}\otimes
\mathrm{id}_A^{\otimes j})(a\otimes  x\otimes b\otimes
y)\\[3pt]
&&\ \ \ \ \ \ \ \ \ +(\lambda\cdot m\otimes
\Join_{\sigma,\lambda})(\mathrm{id}_A\otimes \beta_{i1} \otimes
\mathrm{id}_A^{\otimes j})(a\otimes x\otimes b\otimes y)\Big)\\[3pt]
&=&1_A\underline{\otimes}\big((a\otimes x)\Join_{\sigma,\lambda}(b\otimes y)\big)\\[3pt]
&=&P(a\underline{\otimes}x)\lozenge_{\sigma,\lambda}P(b\underline{\otimes}y).\end{eqnarray*}\end{proof}

Here we mention that if $\sigma$ is the usual flip map, the
resulting Rota-Baxter algebra $(\mathcal{R}_{\sigma,\lambda}(A),
P)$ in the above theorem is the free commutative Rota-Baxter
algebra constructed in \cite{GK} and \cite{Guo1}. Hence this
result provides concrete examples of Rota-Baxter algebras in a
much bigger framework.

Besides the above construction, we can also provide idempotent
Rota-Baxter operators on $\mathcal{R}_{\sigma,\lambda}(A)$ even
$A$ is not unital. In order to construct such operators, we need
the following proposition. For a proof, one can see Theorem 1.1.13
in \cite{Guo2} or verify it directly.

\begin{proposition}Suppose $R$ is an algebra and $R_1$ and $R_2$ are two subalgebras of $R$ such that $R=R_1\oplus R_2$ as vector space. Then the projection $P$ from $R$ onto $R_1$ is an idempotent Rota-Baxter operator of weight -1. \end{proposition}

Note that $\mathcal{R}_{\sigma,\lambda}(A)=A\underline{\otimes
}T(A)=(A\underline{\otimes }\mathbb{K})\bigoplus(\bigoplus_{i\geq
1}A\underline{\otimes}A^{\otimes i})$ as a vector space. Since
$\beta_{01}$ is the identity map and $T_{\sigma}^{qsh}(A)$ is a
filtered algebra, both of $A=A\underline{\otimes }\mathbb{K}$ and
$\bigoplus_{i\geq 1}A\underline{\otimes}A^{\otimes i}$ are
subalgebras of $\mathcal{R}_{\sigma,\lambda}(A)$. Therefore we
have immediately the following.

\begin{theorem}Suppose $(A, m,\sigma)$ is a braided algebra. Then the projection $P_1$ from $\mathcal{R}_{\sigma,\lambda}(A)$ onto its subalgebra $A$ is an idempotent Rota-Baxter operator of weight -1. So is the projection $P_2$ from $\mathcal{R}_{\sigma,\lambda}(A)$ onto its subalgebra $\bigoplus_{i\geq
1}A\underline{\otimes}A^{\otimes i}$. \end{theorem}

Now we turn to tridendriform algebras which were introduced by
Loday and Ronco (\cite{LR}).

\begin{definition}Let $V$ be a vector space, and $\prec$, $\succ$ and $\cdot$ be three binary operations on $V$. The quadruple $(V,\prec, \succ, \cdot)$ is called a \emph{tridendriform algebra} if the following relations are
satisfied: for any $x,y,z\in V$,
\[\begin{split}(x\prec y)\prec z&=x\prec(y\ast z),\\[3pt]
(x\succ y)\prec z&=x\succ (y\prec z),\\[3pt]
(x\ast y)\succ z&=x\succ(y\succ z),\\[3pt]
(x\succ y)\cdot z&=x\succ(y \cdot z),\\[3pt]
(x\prec y)\cdot z&=x\cdot (y\succ z),\\[3pt]
(x\cdot y)\prec z&=x\cdot (y\prec z),\\[3pt]
(x\cdot y)\cdot z&=x\cdot (y \cdot z),\end{split}\]where $x\ast
y=x\prec y+x\succ y+ x\cdot y$.\end{definition}

\begin{theorem}Let $(A,m,\sigma)$ be a braided algebra. We define three operations $\cdot$, $\prec$ and $\succ$ on $T(A)$ recursively by: for any $a,b\in A$ and any $x\in A^{\otimes i}$, $x\in A^{\otimes i}$,
\[\begin{split}(a\otimes x)\cdot (b\otimes y)&=(m\otimes\Join_{\sigma (i,j)} )(\mathrm{id}_A\otimes
\beta_{i,1}\otimes \mathrm{id}_A^{\otimes j-1})(a\otimes x\otimes
b\otimes y),\\[3pt]
(a\otimes x)\prec (b\otimes y)&=a\otimes \big(x\Join_\sigma (b\otimes y)\big),\\[3pt]
(a\otimes x)\succ (b\otimes y)&=(\mathrm{id}_A\otimes
\Join_{\sigma (i+1,j)})(\beta_{i+1,1}\otimes
\mathrm{id}_A^{\otimes j})(a\otimes x\otimes b\otimes
y).\end{split}\]Then $(T(A), \prec,\succ,\cdot)$ is a
tridendriform algebra.
\end{theorem}
\begin{proof}All the verifications are direct. We just need that $\Join_\sigma$ is associative and compatible with the braiding $\beta$ in the sense of braided algebras. For instance, we show the third condition. For any $a\in A$ and $x,y,z\in T(A)$,
\begin{eqnarray*}
\lefteqn{(x\Join_\sigma y)\succ (a\otimes
z)}\\[3pt]
&=&(\mathrm{id}_A\otimes \Join_{\sigma})(\beta_{?,1}\otimes
\mathrm{id}_{T(A)})(\Join_\sigma \otimes \mathrm{id}_A\otimes \mathrm{id}_{T(A)})(x \underline{\otimes} y\underline{\otimes} a\otimes z)\\[3pt]
&=&(\mathrm{id}_A\otimes \Join_{\sigma})(\beta_{?,1}(\Join_\sigma
\otimes \mathrm{id}_A)\otimes
\mathrm{id}_{T(A)})(x \underline{\otimes} y\underline{\otimes} a\otimes z)\\[3pt]
&=&(\mathrm{id}_A\otimes
\Join_{\sigma})(\mathrm{id}_A\otimes\Join_\sigma  \otimes
\mathrm{id}_{T(A)})\\[3pt]
&&\circ(\beta_{?,1}\otimes \mathrm{id}_{T(A)}\otimes\mathrm{id}_{T(A)})(\mathrm{id}_{T(A)}\otimes\beta_{?,1}\otimes \mathrm{id}_{T(A)})(x \underline{\otimes} y\underline{\otimes} a\otimes z)\\[3pt]
&=&(\mathrm{id}_A\otimes
\Join_{\sigma})(\mathrm{id}_A\otimes\mathrm{id}_{T(A)}
\otimes\Join_\sigma
)\\[3pt]
&&\circ(\beta_{?,1}\otimes \mathrm{id}_{T(A)}\otimes\mathrm{id}_{T(A)})(\mathrm{id}_{T(A)}\otimes\beta_{?,1}\otimes \mathrm{id}_{T(A)})(x \underline{\otimes} y\underline{\otimes} a\otimes z)\\[3pt]
&=&(\mathrm{id}_A\otimes \Join_{\sigma})(\beta_{?,1}\otimes \mathrm{id}_{T(A)})\\[3pt]
&&\circ(\mathrm{id}_{T(A)}\otimes\mathrm{id}_A\otimes\Join_\sigma )(\mathrm{id}_{T(A)}\otimes\beta_{?,1}\otimes \mathrm{id}_{T(A)})(x \underline{\otimes} y\underline{\otimes} a\otimes z)\\[3pt]
&=&x\succ\big(y\succ  (a\otimes z)\big).
\end{eqnarray*} Here we denote by $\beta_{?,1}$
the action of $\beta$ on $A^{\otimes ?}\underline{\otimes}A$ where
$?$ is a number determined by the corresponding
computation.\end{proof}

\begin{remark}Let $(R,\cdot,P)$ be a Rota-Baxter algebra of weight 1. Define $a\prec b=a\cdot P(b)$ and $a\succ b=P(a)\cdot b$. Then $(R,\prec,\succ,\cdot)$ is a tridendriform algebra (see \cite{Fa}). Using this fact, we can prove the above theorem by an easier argument: embed $A$ into the unital braided algebra $(\widetilde{A},\widetilde{m},1,\widetilde{\sigma})$, then the tridendriform algebra structure in Theorem 2.8 comes from the Rota-Baxter algebra $(\mathcal{R}_{\widetilde{\sigma},1}(\widetilde{A}),P)$ defined in Theorem 2.4.\end{remark}

\section{Braided Rota-Baxter algebras and their examples}

As we mentioned in the introduction, braided categories play an
important role in mathematical physics. So we would like to extend
the notion of Rota-Baxter algebras in braided categories.

\begin{definition}A triple $(R,P,\sigma)$ is called a \emph{braided Rota-Baxter algebra of weight $\lambda$} if $(R,\sigma)$ is a braided algebra and $P$ is an endomorphism of $R$ such that $(R,P)$ is a Rota-Baxter algebra of weight $\lambda$ and $\sigma(P\otimes P)=(P\otimes P)\sigma$.\end{definition}

\begin{example}Let $(A,m,1_A,\sigma)$ be a unital braided algebra and $(\mathcal{R}_{\sigma,\lambda}(A),P
)$ be the Rota-Baxter algebra defined before. Then
$(\mathcal{R}_{\sigma,\lambda}(A),P,\beta)$ is a braided Rota
baxter of weight $\lambda$.

Indeed, the only thing we need to verify is that $\beta(P \otimes
P )=(P \otimes P )\beta$. For any $a, b\in A$ and $x\in A^{\otimes
i}$,$ y\in A^{\otimes j}$, we have
\begin{eqnarray*}\lefteqn{\beta(P_A\otimes
P )\big((a\underline{\otimes}x)\underline{\otimes}(b\underline{\otimes}y)\big)}\\[3pt]
&=&\beta\Big(\big(1_A\underline{\otimes}(a\otimes x)\big)\underline{\otimes}\big(1_A\underline{\otimes}(b\otimes y)\big)\Big)\\[3pt]
&=&1_A\underline{\otimes}(\beta_{1,j+1}\otimes \mathrm{id}_A^{\otimes i+1})\Big(1_A\underline{\otimes}\beta_{i+1,j+1}\big((a\otimes x)\underline{\otimes}(b\otimes y)\big)\Big)\\[3pt]
&=&(P \otimes P
)\beta\big((a\underline{\otimes}x)\underline{\otimes}(b\underline{\otimes}y)\big).\end{eqnarray*}\end{example}

\begin{example}Let $(A,m,\sigma)$ be a braided algebra and $P_1$ be the projection defined in Theorem 2.6. Then $(\mathcal{R}_{\sigma,\lambda}(A),P_1,\beta)$ is a braided Rota baxter of weight -1.

Again, we only need to verify $\beta(P_1\otimes P_1)=(P_1\otimes
P_1)\beta$. For any $x_1,y_1\in A$ and $x_2,y_2\in
\bigoplus_{i\geq 1}A\underline{\otimes}A^{\otimes
i}$,\begin{eqnarray*}\lefteqn{(P_1\otimes P_1)\beta\big((x_1+x_2)\otimes(y_1+y_2)\big)}\\[3pt]
&=&(P_1\otimes P_1)\big(\beta(x_1\otimes y_1)+\beta(x_1\otimes y_2)+\beta(x_2\otimes y_1)+\beta(x_2\otimes y_2)\big)\\[3pt]
&=&\beta(x_1\otimes y_1)\\[3pt]
&=&\beta(P_1\otimes
P_1)\big((x_1+x_2)\otimes(y_1+y_2)\big),\end{eqnarray*}where the
second equality follows from the fact that $\beta(A^{\otimes
i}\underline{\otimes}A^{\otimes j})\subset A^{\otimes
j}\underline{\otimes}A^{\otimes i}$.

Similarly, $(\mathcal{R}_{\sigma,\lambda}(A),P_2,\beta)$ is also
such an algebra.
\end{example}

\begin{proposition}Let $(R,P)$ be a Rota-Baxter algebra of weight $\lambda$ and $\sigma$ be a braiding on $R$ such that $(R,\sigma)$ is a braided algebra and\[
\begin{split}
\sigma(P\otimes \mathrm{id})&=(\mathrm{id}\otimes P)\sigma,\\[3pt]
\sigma(\mathrm{id}\otimes P)&=(P\otimes \mathrm{id})\sigma.
\end{split}
\]
Then $(R,P,\sigma)$ is a braided Rota-Baxter algebra of weight
$\lambda$. Moreover, if we define $$x\star_P y=xP(y)+P(x)y+\lambda
xy,$$ for any $x,y \in R$, then $(R,\star_P,P,\sigma)$ is again a
braided Rota-Baxter algebra of weight $\lambda$.\end{proposition}
\begin{proof}Note that\begin{eqnarray*}
\sigma(P\otimes P)&=&\sigma(P\otimes \mathrm{id})(\mathrm{id}\otimes P)\\[3pt]
&=&(\mathrm{id}\otimes P)\sigma(\mathrm{id}\otimes P)\\[3pt]
&=&(\mathrm{id}\otimes P)(P\otimes \mathrm{id})\sigma\\[3pt]
&=&(P\otimes P)\sigma.
\end{eqnarray*}
So $(R,P,\sigma)$ is a braided Rota-Baxter algebra of weight
$\lambda$.

It is well-known that $(R,\star_P,P)$ is a Rota-Baxter algebra of
weight $\lambda$. We denote by $m$ the multiplication of $R$. Then
we have
\begin{eqnarray*}\lefteqn{\sigma(\star_P\otimes \mathrm{id})}\\[3pt]
&=&\sigma\big((m\otimes \mathrm{id})( \mathrm{id}\otimes P\otimes \mathrm{id})+(m\otimes \mathrm{id})( P\otimes \mathrm{id}\otimes \mathrm{id})+\lambda m\otimes \mathrm{id}\big)\\[3pt]
&=&(\mathrm{id}\otimes m)\sigma_1\sigma_2( \mathrm{id}\otimes P\otimes \mathrm{id})+(\mathrm{id}\otimes m)\sigma_1\sigma_2( P\otimes \mathrm{id}\otimes \mathrm{id})\\[3pt]
&&+(\mathrm{id}\otimes \lambda m)\sigma_1\sigma_2\\[3pt]
&=&\big((\mathrm{id}\otimes m)( \mathrm{id}\otimes\mathrm{id} \otimes P)+(\mathrm{id}\otimes m)(\mathrm{id} \otimes P\otimes \mathrm{id})+(\mathrm{id}\otimes \lambda m)\big)\sigma_1\sigma_2\\[3pt]
&=&(\mathrm{id}\otimes\star_P )\sigma_1\sigma_2.\end{eqnarray*}

The another condition $\sigma(\mathrm{id}\otimes\star_P
)=(\star_P\otimes \mathrm{id})\sigma_2\sigma_1$ can be verified
similarly.\end{proof}

Given any triple $(R,P,\sigma)$ described as above, this
proposition provides another example of 2-braided algebras.

We provide more examples of braided Rota-Baxter algebras by using
quantum multi-brace algebras introduced in \cite{JR}.

\begin{definition}A \emph{quantum multi-brace algebra} $(V,M,\sigma)$ is a braided vector space $(V,\sigma)$ equipped with a family of operations $M=\{ M_{pq}\}_{p,q\geq 0}$, where $$M_{pq}:V^{\otimes p}\otimes V^{\otimes q}\rightarrow
V,\ \ p\geq 0,\ q\geq 0,$$satisfying

(i)\[
\begin{array}{lllll}
M_{00}&=&0,&& \\
M_{10}&=&\mathrm{id}_V&=&M_{01},\\
M_{n0}&=&0&=&M_{0n},\ \mathrm{for}\ n\geq 2,
\end{array}
\]

(ii) for any $i,j,k\geq 1$,
\[
\begin{split}
\beta_{1k}(M_{ij}\otimes \mathrm{id}_V^{\otimes k})&=(
\mathrm{id}_V^{\otimes k}\otimes M_{ij})\beta_{i+j,k} , \\[3pt]
\beta_{i1}(\mathrm{id}_V^{\otimes i}\otimes
M_{jk})&=(M_{jk}\otimes \mathrm{id}_V^{\otimes i} )\beta_{i,j+k},
\end{split}
\]

(iii) for any triple $(i,j,k)$ of positive integers,
\begin{eqnarray*}
\lefteqn{\sum_{r=1}^{i+j}M_{rk}\circ \big((M^{\otimes r}\circ
\bigtriangleup_\beta^{(r-1)})\otimes
\mathrm{id}_V^{\otimes k}\big)}\\
&=&\sum_{l=1}^{j+k}M_{il}\circ \big(\mathrm{id}_V^{\otimes
i}\otimes( M^{\otimes l}\circ \bigtriangleup_\beta^{(l-1)})\big).
\end{eqnarray*}
\end{definition}

Let $\delta$ be the deconcatenation coproduct on $T(V)$, i.e.,
$$\delta(v_1\otimes\cdots\otimes v_n)=\sum_{i=0}^n(v_1\otimes\cdots\otimes v_i)\underline{\otimes }(v_{i+1}\otimes\cdots\otimes
v_n),$$ and $\varepsilon$ be the projection from $T(V)$ onto
$\mathbb{K}$. We define
$\bigtriangleup_\beta=(\mathrm{id}_{T^c(V)}\otimes \beta\otimes
\mathrm{id}_{T^c(V)})\circ (\delta\otimes\delta)$, and
$$\ast=(\varepsilon\otimes \varepsilon)+\sum_{n\geq 1}M^{\otimes
n}\circ \bigtriangleup_\beta^{(n-1)}:T(V)\underline{\otimes}
T(V)\rightarrow T(V).$$

\begin{proposition}[\cite{JR}]Let $(V,M,\sigma)$ be a quantum multi-brace algebra. Then $(T(V), \ast,\beta)$ is a braided algebra. \end{proposition}

From now on, we focus on the category of Yetter-Drinfeld modules.
Let $H$ be a Hopf algebra. The category $^H_H\mathcal{YD}$ of
Yetter-Drinfeld modules over $H$ consists of the following data:
objects in $^H_H\mathcal{YD}$ are vector spaces $V$ having
simultaneously a left $H$-module structure $\cdot$ and a left
$H$-comodule structure such that whenever $h\in H$ and $v\in V$,
$$\sum h_{(1)}v_{(-1)}\otimes h_{(2)}\cdot v_{(0)}=\sum
(h_{(1)}\cdot v)_{(-1)}h_{(2)}\otimes (h_{(1)}\cdot v)_{(0)},
$$ and morphisms in $^H_H\mathcal{YD}$ are linear maps which are
module and comodule homomorphisms. It is well-known that
$^H_H\mathcal{YD}$ is a braided tensor category. Let $V$ be an
object in $^H_H\mathcal{YD}$. Then $V$ is a braided vector space
with the natural braiding defined by $\sigma(v\otimes w)=\sum
v_{(-1)}\cdot w\otimes v_{(0)}$. An algebra in the category
$^H_H\mathcal{YD}$ is an object $A$ in $^H_H\mathcal{YD}$ together
with an associative multiplication $m$ which is a morphism in this
category. Clearly, in this case, $(A,m)$ is a braided algebra with
respect to the natural braiding.

\begin{proposition}Let $V$ be an object in $^H_H\mathcal{YD}$ and $\sigma$ be its natural braiding. If $M$ is a quantum multi-brace algebra structure on $(V,\sigma)$
such that all $M_{pq}$ are morphism in $^H_H\mathcal{YD}$, then
$(T(V), \ast)$ is an algebra in $^H_H\mathcal{YD}$.
\end{proposition}
\begin{proof}Since both of the $H$-module and $H$-comodule structures are diagonal, and all maps involving in the formula of $\ast$ are module and comodule homomorphisms, the result follows from an easy verification.\end{proof}

Consider the bosonization $T(V)\#H$ of $(T(V),\ast)$ by $H$. It is
an associative algebra with underlying vector space $T(V)\otimes
H$ and multiplication defined by $$(x\#h)(y\#h')=\sum
x\ast(h_{(1)}\cdot y)\# h_{(2)}h', x,y\in T(V), h,h'\in H.$$Here,
we use the symbol $x\#h$, instead of $x\otimes h$, to indicate the
new structure.

Furthermore, $T(V)\#H$ is a braided algebra with respect to the
natural braiding $\Sigma$ defined by
$$\Sigma\big((x\#h)\otimes(y\#h')\big)=\sum ((x_{(-3)}h_{(1)})\cdot y\#x_{(-2)}h_{(2)} h'S(x_{(-1)}h_{(3)}))\otimes (x_{(0)}\#h_{(4)}),
$$where $S$ is the antipole of $H$.

Observe that both of $H=\mathbb{K}\# H$ and $T^+(V)\#H$ are
subalgebras of $T(V)\#H$ and $T(V)\#H=H\oplus T^+(V)\#H$. By
Proposition 2.5, the projection $P_0$ from $T(V)\#H$ onto $H$ is
an idempotent Rota-Baxter operator of weight -1. In addition,
since $\Sigma\big((V^{\otimes p}\# H)\otimes (V^{\otimes q}\#
H)\big)\subset (V^{\otimes q}\# H)\otimes (V^{\otimes p}\# H)$, we
obtain that $\Sigma(P_0\otimes \mathrm{id})=(\mathrm{id}\otimes
P_0)\Sigma$ and $ \Sigma(\mathrm{id}\otimes P_0)=(P_0\otimes
\mathrm{id})\Sigma$. By summarizing the above discussion and
Proposition 3.4, we have

\begin{theorem}The triple $(T(V)\#H, P_0, \Sigma)$ is a braided Rota-Baxter algebra of weight -1. \end{theorem}

\begin{remark}The algebra $T(V)\#H$ is isomorphic, as a Hopf algebra, to the quantum multi-brace cotensor Hopf algebra introduced in \cite{FR}.\end{remark}

\section*{Acknowledgements}
The author is grateful to Marc Rosso for his constant
encouragement. He would like to thank Daniel Sternheimer for his
suggestion on a previous title of this paper. Finally, he would
like to thank the referees for their careful reading and useful
comments, especially, thank one of them for drawing his attention
to the Connes-Kreimer theory. This work was partially supported by
The National Natural Science Foundation of China (Grant No.
11201067).

\bibliographystyle{amsplain}

\end{document}